\documentclass[letterpaper, reqno, 11pt]{amsart}
\usepackage{amssymb, amscd, latexsym, amsfonts, mathrsfs}
\usepackage[english]{babel}
\RequirePackage[utf8]{inputenc}

\usepackage{enumerate, cite, calc, xstring, hyperref}
\usepackage{float, wrapfig, sidecap, multicol, needspace, booktabs}

\usepackage[T1]{fontenc}
\usepackage[all]{xy}
\usepackage{graphicx, xcolor, url}
\usepackage[top=2.75cm, bottom=2.75cm, left=2.75cm, right=2.75cm]{geometry}

\usepackage[T1]{fontenc}
\usepackage[scaled=.90]{helvet} % \textss style substitute
\usepackage{dsfont, upgreek, mathabx, yfonts}
\usepackage{relsize}

\newtheorem{theorem}{Theorem}[section]
\newtheorem{proposition}[theorem]{Proposition}
\newtheorem{corollary}[theorem]{Corollary}
\newtheorem{lemma}[theorem]{Lemma}

\theoremstyle{definition}
\newtheorem{definition}[theorem]{Definition}
\newtheorem{example}[theorem]{Example}

\theoremstyle{remark}
\newtheorem{remark}[theorem]{Remark}

\numberwithin{equation}{section}

\newcommand{\tr}{\mathrm{tr}\,}

\newcommand{\mr}[1]{\mathrm{#1}}
\newcommand{\mc}[1]{\mathcal{#1}}

\newcommand{\bs}[1]{\boldsymbol{#1}}

\newcommand{\cx}{\bar{x}}

\newcommand{\bze}{\bs{0}}

\newcommand{\bb}{\bs{s}}
\newcommand{\bx}{\bs{x}}
\newcommand{\by}{\bs{y}}
\newcommand{\bz}{\bs{z}}
\newcommand{\bu}{\bs{u}}

\newcommand{\bH}{\bs{H}}

\newcommand{\bE}{\bs{e}}
\newcommand{\bN}{\bs{N}}
\newcommand{\bR}{\bs{R}}
\newcommand{\Rie}[4]{\langle\bR({#1},{#2}){#3},\,{#4}\rangle}
\newcommand{\Ric}[2]{\bs{\mc{R}ic}({#1},{#2})}

\newcommand{\cR}{\mc{R}}

\newcommand{\II}{\bs{\mr{II}}}

\newcommand{\metric}[2]{\langle\, {#1},\,{#2}\,\rangle}

\renewcommand{\hat}{\widehat}
\renewcommand{\bar}{\overline}

\newcommand{\dd}{\partial}

\newcommand{\del}{\delta}

\newcommand{\al}{\alpha}
\newcommand{\bet}{\beta}
\newcommand{\vep}{\varepsilon}
\newcommand{\lbd}{\lambda}

\newcommand{\ka}{\kappa}

\newcommand{\diag}{\mathrm{diag}}
\def\RR{{\mathbb R}}

\def\SS{{\mathbb S}}

\def\cS{{\mathcal S}}

\def\cO{{\mathcal O}}
\def\cM{{\mathcal M}}

\newcommand{\Id}{\mathrm{Id}}

\begin{document}

\title{Manifold Curvature Descriptors from Hypersurface Integral Invariants}

\author[J. \'Alvarez-Vizoso]{Javier \'Alvarez-Vizoso}
\author[M. Kirby]{Michael Kirby}
\author[C. Peterson]{Chris Peterson}
\address{Department of Mathematics, Colorado State University, Fort Collins, CO, USA}
\email{alvarez@math.colostate.edu, kirby@math.colostate.edu, peterson@math.colostate.edu}

\date{\today}

\maketitle

%%%%%%%%%%%%%%%%%%%%%%%%%%%%%%%%%%%%%%%%%%%%%%%%%%%%%%%%%%%%%%%%%%%%%%
%%%%%%%%%%%%%%%        ABSTRACT
%%%%%%%%%%%%%%%%%%%%%%%%%%%%%%%%%%%%%%%%%%%%%%%%%%%%%%%%%%%%%%%%%%%%%%

% REQUIRED
\begin{abstract}
Integral invariants obtained from Principal Component Analysis on a small kernel domain of a submanifold encode important geometric information classically defined in differential-geometric terms. We generalize to hypersurfaces in any dimension major results known for surfaces in space, which in turn yield a method to estimate the extrinsic and intrinsic curvature of an embedded Riemannian submanifold of general codimension. In particular, integral invariants are defined by the volume, barycenter, and the EVD of the covariance matrix of the domain. We obtain the asymptotic expansion of such invariants for a spherical volume component delimited by a hypersurface and for the hypersurface patch created by ball intersetions, showing that the eigenvalues and eigenvectors can be used as multi-scale estimators of the principal curvatures and principal directions. This approach may be interpreted as performing statistical analysis on the underlying point-set of a submanifold in order to obtain geometric descriptors at scale with potential applications to Manifold Learning and Geometry Processing of point clouds.
\end{abstract}

%%%%%%%%%%%%%%%%%%%%%%%%%%%%%%%%%%%%%%%%%%%%%%%%%%%%%%%%%%%%%%%%%%%%%%
%%%%%%%%%%%%%%%        TOC customization

% use [4] if hyperref is not in use, [5] otherwise
\DeclareRobustCommand{\SkipTocEntry}[5]{}

\makeatletter
\def\@tocline#1#2#3#4#5#6#7{\relax
\ifnum #1>\c@tocdepth % then omit
  \else 
    \par \addpenalty\@secpenalty\addvspace{#2}% 
\begingroup \hyphenpenalty\@M
    \@ifempty{#4}{%
      \@tempdima\csname r@tocindent\number#1\endcsname\relax
 }{%
   \@tempdima#4\relax
 }%
 \parindent\z@ \leftskip#3\relax \advance\leftskip\@tempdima\relax
 \rightskip\@pnumwidth plus4em \parfillskip-\@pnumwidth
 #5\leavevmode\hskip-\@tempdima #6\nobreak\relax
 \ifnum#1<0\hfill\else\dotfill\fi\hbox to\@pnumwidth{\@tocpagenum{#7}}\par
 \nobreak
 \endgroup
  \fi}
\makeatother

\tableofcontents

%%%%%%%%%%%%%%%%%%%%%%%%%%%%%%%%%%%%%%%%%%%%%%%%%%%%%%%%%%%%%%%%%%%%%%
%%%%%%%%%%%%%%%%%%%%%%%%%%%%%%%%%%%%%%%%%%%%%%%%%%%%%%%%%%%%%%%%%%%%%%
%%%%%%%%%%%%%%%			0. INTRODUCTION
%%%%%%%%%%%%%%%%%%%%%%%%%%%%%%%%%%%%%%%%%%%%%%%%%%%%%%%%%%%%%%%%%%%%%%
%%%%%%%%%%%%%%%%%%%%%%%%%%%%%%%%%%%%%%%%%%%%%%%%%%%%%%%%%%%%%%%%%%%%%%

\section{Introduction}\label{sec:intro}

Manifold Learning has as its prime goal the local characterization and reconstruction of manifold geometry from the study of the underlying point set, usually embedded as a submanifold in an ambient space, typically Euclidean. To obtain theoretical results that can serve as tools for this endeavour, it is assumed that the complete continuous point set is known so that local statistical invariants on given domains can be shown to relate to the relevant local geometry, whereas in practice only a finite cloud of points, probably with noise, is available. In Geometry Processing, the development of these methods provides us with descriptors that function as geometry estimators, guide a possible reconstruction or serve as feature detectors. The integral invariant point of view attempts to overcome some of the difficulties of computational geometry when facing the task of extracting information that is classically defined as a differential invariant, like curvature, since its discrete version reduces to, e.g., sums instead of finite differences. The multi-scale behaviour and averaging nature of these invariants is also of importance in applications and their possible stability and robustness with respect to noise.

Series expansion of the volume of small geodesic balls within a manifold \cite{gray1979}, and volumes cut out by a hypersurface inside a ball of the ambient space \cite{hulin2003}, have been shown to be given in terms of the manifold curvature scalar invariants. In order to obtain local adaptive Galerkin bases for large-dimensional dynamical systems, the eigenvalue decomposition of covariance matrices of spherical intersection domains on the invariant manifold was introduced in \cite{broomhead1991local}, \cite{solis1993, solis2000} to provide estimates of the dimension of the manifold and a suitable decomposition of phase space at every point. In the case of curves, the Frenet-Serret apparatus is recovered with explicit formulas at scale to obtain descriptors of the generalized curvatures in terms of the eigenvalues of the covariance matrix \cite{alvarez2017}. Integral invariants were already introduced and employed in Geometry Processing applications by \cite{connolly1986},\cite{cazals2003a, cazals2003b}, \cite{clarenz2003, clarenz2004}, \cite{manay2004, manay2006}. Principal Component Analysis performed with covariance matrix integral invariants have been studied primarily for the case of curves and surface in 2D and 3D in \cite{alliez2007}, \cite{berkmann1994}, \cite{clarenz2003, clarenz2004}, \cite{hoppe1992}, \cite{manay2004}, \cite{pottmann2006}, as a means to determine relevant local geometric information while maintaining stability with respect to noise \cite{lai2009}, \cite{pottmann2007, pottmann2008}, e.g., for feature and shape detection using point clouds. Voronoi-based feature estimation \cite{merigot2009, merigot20011} has also taken advantage of the PCA covariance matrix approach. Those methods study embedded manifolds whereas intrinsic probability and statistical analysis using geometric measurements inside a Riemannian manifold have also been developed \cite{pennec1999, pennec2006} and could be use to do covariance analysis of submanifolds embedded in curved ambient spaces.

In this paper we follow and generalize the major theoretical results of \cite{pottmann2007} for surfaces in space to hypersurfaces in any dimension, which in turn allows for the extension of their approach to obtain  descriptors of the extrinsic and intrinsic curvature at a given scale for any Riemannian submanifold of general codimension in Euclidean space. 

The structure of the paper is as follows: in section \S\ref{sec:overview} notation is set with a brief overview of how the curvature of submanifolds, especially hypersurfaces, is classically defined as a differential invariant. In section \S\ref{sec:IntInv} PCA integral invariants and geometric descriptors are introduced to show how the study of hypersurfaces is sufficient to be extended to Riemannian submanifolds of any dimension by applying the analysis onto as many hypersurface projections of the submanifold as its codimension. In section \S\ref{sec:sphVol} these integral invariants are computed for a volume region delimited by a hypersurface inside a ball; the asymptotic expansions of the invariants with respect to the scale of the ball are shown to be given in terms of the principal curvatures and the dimension, and the eigenvectors of the covariance matrix are shown to converge in the limit to the principal directions as well. In section \S\ref{sec:sphPatch} the analogous analysis is carried out for the integral invariants of the hypersurface patch cut out by the ball. In section \S\ref{sec:descrip} we see how these asymptotic formulas can be inverted to yield geometric descriptors at scale of the principal curvatures and principal directions for hypersurfaces, thus establishing concrete formulas to use in the general method outlined for Riemannian submanifolds.

Since these results show that moment-based integral invariants asymptotically encode the classical curvature invariants, they furnish a local multi-scale method to probe the underlying geometry of embedded clouds of data points in higher dimensions using integration over discretized domains instead of differential approximations, thus extending the benefits of using the integral invariant techniques in 3D studied in the literature to the fields of Manifold Learning and Geometry Processing in general dimension.

%%%%%%%%%%%%%%%%%%%%%%%%%%%%%%%%%%%%%%%%%%%%%%%%%%%%%%%%%%%%%%%%%%%%%%
%%%%%%%%%%%%%%%%%%%%%%%%%%%%%%%%%%%%%%%%%%%%%%%%%%%%%%%%%%%%%%%%%%%%%%
%%%%%%%%%%%%%%%			1. REVIEW OF HYPERSURFACES
%%%%%%%%%%%%%%%%%%%%%%%%%%%%%%%%%%%%%%%%%%%%%%%%%%%%%%%%%%%%%%%%%%%%%%
%%%%%%%%%%%%%%%%%%%%%%%%%%%%%%%%%%%%%%%%%%%%%%%%%%%%%%%%%%%%%%%%%%%%%%

\section{Review of Submanifold Differential Geometry}\label{sec:overview}

Here we briefly revisit the basic differential-geometric concepts \cite{chavel2006}, \cite{kobayashi1969}, \cite{oneil1983}, \cite{spivak1999} needed for the present work, and in the appendix we set the notation for spherical coordinates, areas of spheres, volumes of balls and integrals of monomials over them in $\RR^n$.

An embedded $n$-dimensional Riemannian submanifold $(\cM,g)$ in $\RR^{n+k}$ is a smooth submanifold endowed with the metric given by the restriction of the Euclidean metric to the tangent space $T_p\cM$ at every point $p\in\cM$, $g(\bx,\by)=\metric{\bx}{\by}$, also called the first fundamental form $\text{I}(\bx,\by)$. If $\bs{X}:U\subset\RR^n\rightarrow\RR^{n+k}$ is a chart or parametrization of $\cM$ with coordinates $x_\mu$, the tangent space $T_p\cM$ is spanned by the vectors $\dd_\mu\bs{X}\vert_p$, so the metric components in this basis are 
$$g_{\mu\nu}(p)=\metric{\dd_\mu\bs{X}\vert_p}{\dd_\nu\bs{X}\vert_p},$$
and the $n$-dimensional volume element on $\cM$ can be found in this coordinates to be 
$$\text{dVol}=\sqrt{\det g}\; d^n\bx.$$
Euclidean space has a directional derivative given by the usual $\bar{\nabla}_{\bx}$, and when $\bx\in T\cM$ its projection to the manifold tangent space induces the Levi-Civita connection $\nabla_{\bx}$ associated to $g$ on it. The second fundamental form of $\cM$ is defined to be the normal component in ambient space:
$$
\II (\bx ,\by)=(\bar{\nabla}_{\bx}\,\by)_\perp=\bar{\nabla}_{\bx}\,\by-\nabla_{\bx}\,\by,\quad\quad \bx,\by\in T\cM.
$$
It is a symmetric bilinear form on the tangent space with values in the normal bundle of $\cM$ that encodes the extrinsic curvature of the embedded manifold. In particular, the mean curvature vector is defined by the invariance of the trace of $\II$ for any orthonormal tangent frame $\{\bE_\mu\}_{\mu=1}^n$:
$$
\bH = \sum_{\mu=1}^n\;\II(\bE_\mu,\bE_\mu) = \sum_{j=1}^k H_j\bN_j.
$$
The Riemann curvature tensor is the fundamental intrinsic invariant of $\cM$. The standard definition is as a linear operator measuring the non-commutativity of the covariant derivative:
$$
\bR(\bx,\by)\bz=(\nabla_{\bx}\nabla_{\by}-\nabla_{\by}\nabla_{\bx}-\nabla_{[\bx,\by]})\bz,\quad\quad \bx,\by,\bz\in T_p\cM.
$$
The Riemann tensor vanishes locally if and only if there is a chart in whose coordinates the metric tensor reduces to the Euclidean metric, i.e., if the manifold is locally flat, or equivalently, if parallel transport is integrable and there is no geodesic deviation. Contractions of it produce all the other curvature tensors, like the Ricci tensor 
$$
\Ric{\bx}{\by}=\sum_{\mu=1}^n\Rie{\bE_\mu}{\bx}{\by}{\bE_\mu},
$$
and the scalar curvature: $\cR = \sum_{\mu}\Ric{\bE_\mu}{\bE_\mu}$.

Of major importance for the geometry of submanifolds is Gau{\ss} equation.
\begin{theorem}
	Let $\cM$ be an embedded Riemannian manifold in Euclidean space, then the intrinsic Riemann curvature tensor and the extrinsic second fundamental form are related by:
\begin{equation}\label{eq:Gauss}
	\Rie{\bE_\mu}{\bE_\nu}{\bE_\al}{\bE_\bet} =\metric{\II(\bE_\mu,\bE_\bet)}{\II(\bE_\nu,\bE_\al)} - \metric{\II(\bE_\mu,\bE_\al)}{\II(\bE_\nu,\bE_\bet)}.
\end{equation}
for any orthonormal tangent frame $\{\bE_\mu\}_{\mu=1}^n$.
\end{theorem}

A hypersurface $\cS$ is an embedded manifold of codimension 1, many of whose properties generalize those of surfaces in $\RR^3$. Its second fundamental form can also be introduced via the Weingarten map, or shape operator $\bs{\hat{S}}$ defined as follows: given a choice of unit normal vector field $\bN$ around $p\in\cS$, there is a linear endomorphism of $T_p\cS$ given by
$$
\bs{\hat{S}}(\bx) = -\bar{\nabla}_{\bx}\bN,\quad \forall\bx\in T_p\cM
$$
such that the classical second fundamental form is related to the one defined above by:
$$\text{II}(\bx ,\by)=\langle\,\II (\bx ,\by),\,\bN\,\rangle = \metric{\bs{\hat{S}}(\bx)}{\by}.$$
The Weingarten map encodes how the hypersurface normal vector varies in the ambient space when moving in a direction tangent to the hypersurface, thus measuring curvature. Moreover, $\bs{\hat{S}}$ is self-adjoint with respecto to the metric so there is an orthonormal basis of $T_p\cS$ given by its eigenvectors called the principal directions of $\cS$ at $p$. The corresponding eigenvalues are called principal curvatures, $\{\ka_\mu(p)\}_{\mu=1}^n$, because $\metric{\bs{\hat{S}}(\bu)}{\bu}$ measures the normal acceleration of a curve inside $\cS$ with unit tangent $\bu$. The $2$-plane spanned by a tangent vector $\bu\in T_p\cS$ and the normal vector $\bN\in N_p\cS$ intersects the hypersurface in a normal section curve whose first Frenet-Serret curvature is precisely the normal curvature given by $\metric{\bs{\hat{S}}(\bu)}{\bu}$. For any tangent vector the normal section curvature is
\begin{equation}\label{eq:normalCurv}
	\ka(\bx)=\frac{\text{II}(\bx,\bx)}{\text{I}(\bx,\bx)},\quad \forall \bx\in T_p\cS
\end{equation}

Furthermore, one can define elementary curvature scalars $K_1(p),\dots, K_n(p)$ as the elementary symmetric polynomials on the $\{\ka_\mu(p)\}_{\mu=1}^n$. In particular the mean curvature of a hypersurface is
\begin{equation}
	H(p)=K_1(p)=\tr\bs{\hat{S}} =\sum_{\mu=1}^n\ka_\mu(p) = \sum_{\mu=1}^n\text{II}(\bE_\mu,\bE_\mu),
\end{equation}
the scalar curvature is $\cR(p) = 2K_2(p)$, and the Gau{\ss}ian curvature is 
\begin{equation}
	K_n(p) =\det\bs{\hat{S}}= \prod_{\mu=1}^n\ka_\mu(p) = \prod_{\mu=1}^n\text{II}(\bE_\mu,\bE_\mu).
\end{equation}
\begin{remark}
	To simplify notation we shall write $\ka_\mu, H,\,\cR$ instead of $\ka_\mu(p), H(p),\,\cR(p)$, etc. if the point is understood from the context. The point itself may be denoted $p$ if interpreted set-theoretically in $\cS$, or $\bs{p}$ if considered as a vector when it appears in linear operations of $\RR^{n+1}$.
\end{remark}
Notice the most elementary Newton relation between the power sum function of order 2 and the elementary symmetric polynomials yields the useful expression:
\begin{equation}\label{eq:trS2}
\tr\bs{\hat{S}}^2=\sum_{\mu=1}^n\ka^2_\mu = K_1-2K_2 = H^2-\cR.
\end{equation}
In fact, more is true since Gau{\ss} equation applied to the Ricci tensor of a hypersurface leads to
\begin{equation}\label{eq:iiiForm}
	\Ric{\bx}{\by}=H\metric{\bs{\hat{S}}(\bx)}{\by}-\metric{\bs{\hat{S}}^2 (\bx)}{\by}.
\end{equation}

Using the implicit function theorem, the wedge product of the tangent basis and Gau{\ss} equation, one can find the following crucial lemma for the approximations made in the present work.
\begin{lemma}\label{lem:localHyper}
There is an open neighborhood $U_p$ around any point $p\in\cS$ such that the smooth hypersurface $\cS$ is locally given by a graph $z:U_p\subset T_p\cS\cong\RR^n\rightarrow T_p\cS\oplus\langle\bN_p\rangle \cong\RR^{n+1}$, with $\bs{p}=\bze$, and $\nabla z(\bze)=\bze$, thus defined to leading order by an osculating paraboloid which in the basis of principal directions becomes:
\begin{equation}
z(\bx) = \frac{1}{2}\sum_{\mu=1}^n\ka_\mu x^2_\mu +\cO(x^3).
\end{equation}
In this neighborhood the area element is
\begin{equation}\label{eq:volElem}
	\text{dVol}\vert_{U_p}=\sqrt{\det g}\;\,d^n\bx = \sqrt{1+\sum_{\mu=1}^n\left(\frac{\dd z}{\dd x_\mu}\right)^2}\; dx_1\cdots dx_n.
\end{equation}
The second fundamental form at $p$ corresponds to the Hessian matrix of $z(\bx)$ at $p$:
\begin{equation}
\text{\emph{II}}_p(\bE_\mu,\bE_\nu) = \left[ \frac{\dd^2 z}{\dd x_\mu\dd x_\nu} \right]_p
\end{equation}
where $\{\bE_\mu\}_{\mu=1}^n$ are the principal basis vectors. In this basis the Riemann tensor reduces to
\begin{equation}
	\Rie{\bE_\mu}{\bE_\nu}{\bE_\al}{\bE_\bet}= \ka_\mu(p)\ka_\nu(p)(\del_{\al\nu}\del_{\mu\bet}-\del_{\al\mu}\del_{\bet\nu}),
\end{equation}
the diagonal components of the Ricci tensor are
\begin{equation}
	\Ric{\bE_\mu}{\bE_\mu} =R_{\mu\mu}(p)= \sum_{\al\neq\mu}^n\ka_\al(p)\ka_\mu(p),
\end{equation}
and the scalar curvature is
\begin{equation}
	\cR(p) = 2K_2(p) = 2\sum_{\mu <\nu}^n\ka_\mu(p)\ka_\nu(p).
\end{equation}
\end{lemma}

%%%%%%%%%%%%%%%%%%%%%%%%%%%%%%%%%%%%%%%%%%%%%%%%%%%%%%%%%%%%%%%%%%%%%%
%%%%%%%%%%%%%%%%%%%%%%%%%%%%%%%%%%%%%%%%%%%%%%%%%%%%%%%%%%%%%%%%%%%%%%
%%%%%%%%%%%%%%%			2. INTEGRAL INVARIANTS & DESCRIPTORS
%%%%%%%%%%%%%%%%%%%%%%%%%%%%%%%%%%%%%%%%%%%%%%%%%%%%%%%%%%%%%%%%%%%%%%
%%%%%%%%%%%%%%%%%%%%%%%%%%%%%%%%%%%%%%%%%%%%%%%%%%%%%%%%%%%%%%%%%%%%%%

\section{Integral Invariants and Descriptors}\label{sec:IntInv}

Our approach generalizes the theoretical part of the seminal work \cite{pottmann2007} with a focus on the analytical expansion of integral invariants to get descriptors of manifold curvature in any dimension. The local integral invariants considered are integrals over small kernel domains determined by balls and the hypersurface. In particular, we will focus on the Principal Component Analysis of a $(n+1)$-dimensional region delimited by the hypersurface inside a ball centered at a point on the hypersurface, and the $n$-dimensional patch on the hypersurface cut out by such a ball. In general, one can define invariants for a measurable domain by computing the moments of the coordinates of the points inside, which leads us to

\begin{definition}
	Let $D$ be a measurable domain in $\RR^n$, the \emph{integral invariants} associated to the moments of order 0, 1 and 2 of the coordinate functions of the points of $D$ are: the volume
	\begin{equation}
		V(D) = \int_D 1\; \text{dVol},
	\end{equation}
	the barycenter
	\begin{equation}
		\bb(D) = \frac{1}{V(D)}\int_D \bs{X} \; \text{dVol},
	\end{equation} 
	and the eigenvalue decomposition of the covariance matrix
	\begin{equation}
		C(D) =\int_D (\bs{X}-\bb(D))\otimes (\bs{X}-\bb(D))^T \; \text{dVol}.
	\end{equation} 
	Here dVol is the measure on $D$ induced by restriction of the Euclidean measure, and the tensor product is to be understood as the outer product of the components in a chosen basis.
\end{definition}

These can be interpreted as statistical characterization measurements of a continuous distrubition: the volume measures the size or mass of the set; the barycenter measures the centralization of the domain as a mean or average point, i.e., a center of mass; finally, the covariance matrix is a measure of the dispersion of the points in $D$ around its center of mass. From this statistical point of view, we could define the covariance matrix normalized by $V(D)$ as well, so that $\frac{\text{dVol}}{V}$ is a density, but this will not affect our results in any significant way (essentially, the second-to-leading order term in the volume equations \ref{eq:volSph}, \ref{eq:volPatch}, would get added to the eigenvalues at that order).

An \emph{integral invariant descriptor} $F(D)$ of some feature $F$ of a measurable domain $D$ is any expression for $F$ completely given in terms of $V(D),\, \bb(D)$, the eigenvalue decomposition of $C(D)$ or other integral invariants. If the domain $D$ is determined by a region of a hypersurface $\cS$, the main geometric descriptors are any principal curvature estimators $\ka_\mu(D)$ of $\ka_\mu(p)$, and principal and normal direction estimators $\bE_\mu(D),\,\bN(D)$ of $\bE_\mu(p),\bN(p)$, for some known point $p\in\cS$. If the domain $D$ is determined by a region of an embedded manifold $\cM$, the main geometric descriptor is any second fundamental form estimator, $\II(D)$ of $\II_p$, for some known point $p\in\cM$. Since our domain $D$ of interest will possess a natural scale $\vep$ determined by the size of the ball that shall define it, we shall talk about \emph{descriptors at scale}. Moreover, throughout all the paper we consider $\vep$ to be small enough so that we can approximate the hypersurface $\cS$ by the local graph representation of its osculating paraboloid at $p$, which is sufficient to obtain the first terms of the asymptotic expansions of the integral invariants.

These descriptors become valuable tools to perform Manifold Learning, feature detection and shape estimation when only partial knowledge of the complete set of points is known or when noise is present. In this regard, \cite{pottmann2006, pottmann2007, pottmann2008} carried out experimental and theoretical analysis of the stability of these and other descriptors in the case of curves and surfaces in $\RR^3$, reporting for example that the invariants of the spherical component domain are more robust with respect to noise than the patch region ones. It is to be expected that the same stability behavior holds in the hypersurface case due to the sensitivity to small changes of an $n$-dimensional patch compared to an $(n+1)$-dimensional volume of which the perturbed patch is part of its boundary.

When the asymptotic expansions with respect to scale of hypersurface integral invariants are available to high enough order, curvature information can be extracted by truncating the series and inverting the relations in order to obtain a computable multi-scale estimator of the actual curvatures. In particular, the eigenvalues of the covariance matrix will provide such a descriptor for the principal curvatures of a smooth hypersurface, $\ka_\mu(D)$, and its eigenvectors $\{\bE_\mu(D)\}_{\mu=1}^n$, and $\bE_{n+1}(D)$, will do the same for the principal normal directions. In order to produce analogous descriptors for an embedded Riemannian manifold of higher codimension, we just need to apply the procedure to the $k$ hypersurfaces created by projecting the manifold down to $(n+1)$ linear subspaces.

\begin{lemma}
	Let $\cM\subset\RR^{n+k}$ be an $n$-dimensional embedded Riemannian manifold, and fix an orthonormal tangent basis $\{\bE_\mu\}_{\mu=1}^n$ of the tangent space $T_p\cM$, and an orthonormal basis $\{\bN_j\}_{j=1}^k$ of the normal space $N_p\cM$ at $p\in\cM$. Consider a ball $B^{(n+k)}_p(\vep)$ for small enough $\vep>0$, such that the projections of $\cM\cap B^{(n+k)}_p(\vep)$ onto the linear subspaces $T_p\cM\oplus\langle\bN_j\rangle$, for all $j=1,\dots,k$, are smooth hypersurfaces $\cS_j$. Then, if $\ka^{(j)}_\mu(D),\, \{\bE^{(j)}_\mu(D)\}_{\mu=1}^n$ are descriptors of the principal curvatures and principal directions at $p$ for each of the hypersurfaces $\cS_j$, then the second fundamental form of $\cM$ at $p$ has a descriptor:
\begin{equation}
	\II_p(D)(\bE_\mu,\bE_\nu) = \sum_{j=1}^k [V_j(D) K_j(D) V(D)^T_j]_{\mu\nu}\; \bN_j\; ,
\quad\quad\, \mu,\nu =1,\dots,n,
\end{equation}
where $[V_j(D)]$ are the matrices whose columns are the components of $\{\bE^{(j)}_\mu(D)\}_{\mu=1}^n$ in the chosen basis $\{\bE_\mu\}_{\mu=1}^n$, and $[K_j(D)]$ is the diagonal matrix of principal curvature estimators. In turn, the Riemann curvature tensor of $\cM$ at $p$ acquires a descriptor:
\begin{equation}
	\langle\bR(D)({\bE_\mu},\bE_\nu)\bE_\al,\,\bE_\bet\rangle=\sum_{j=1}^k\left(\,[ V_j K_j V^T_j]_{\mu\bet}[ V_j K_j V^T_j]_{\nu\al}- [V_j K_j V^T_j]_{\mu\al}[V_j K_j V^T_j]_{\nu\bet}\, \right).
\end{equation}
\end{lemma}
\begin{proof}
By the implicit function theorem, there is a neighborhood of $U_p\subset T_p\cM$ such that the manifold can be locally given by a graph 
$
\bx\mapsto (\bx,f_1(\bx),\dots,f_k(\bx)),
$
where $\bx\in U_p$, $p$ corresponds to $0$, and $\nabla f_j(\bze)=\bze$. From this, the projection hypersurfaces $\cS_j$ are just $(\bx,f_j(\bx))$ for $j=1,\dots, k$. It can be shown \cite[vol. II ex. 3.3.]{kobayashi1969} that the second fundamental form of $\cM$ at $p$ is precisely the linear combination of the second fundamental forms of each of the hypersurface projections weighed by the corresponding normal vector, i.e.,
$$
\II_p(\bE_\mu,\bE_\nu) =\sum_{j=1}^k \left[ \frac{\dd^2 f_j}{\dd x_\mu\dd x_\nu}(p) \right]\,\bN_j
$$
Analyzing each of those hypersurfaces in $T_p\cM\oplus\langle\bN_j\rangle\cong\RR^{n+1}$, to obtain descriptors $\ka^{(j)}_\mu(D)$, $\{\bE^{(j)}_\mu(D)\}_{\mu=1}^n$ for every $j$, we obtain precisely a descriptor of the eigenvalue decomposition of each Hessian, i.e., $\text{Hess}\, f_j\vert_p(D)=[V_j(D) K_j(D) V(D)^T_j]$ is an estimator of the second fundamental form of $\cS_j$ at $p$ in the original basis. Applying Gau{\ss} equation \ref{eq:Gauss} yields a corresponding descriptor for the Riemann tensor.
\end{proof}

%%%%%%%%%%%%%%%%%%%%%%%%%%%%%%%%%%%%%%%%%%%%%%%%%%%%%%%%%%%%%%%%%%%%%%
%%%%%%%%%%%%%%%%%%%%%%%%%%%%%%%%%%%%%%%%%%%%%%%%%%%%%%%%%%%%%%%%%%%%%%
%%%%%%%%%%%%%%%			3. VOLUME
%%%%%%%%%%%%%%%%%%%%%%%%%%%%%%%%%%%%%%%%%%%%%%%%%%%%%%%%%%%%%%%%%%%%%%
%%%%%%%%%%%%%%%%%%%%%%%%%%%%%%%%%%%%%%%%%%%%%%%%%%%%%%%%%%%%%%%%%%%%%%

\section{Hypersurface Spherical Component Integral Invariants}\label{sec:sphVol}

The following domain is introduced in \cite{hulin2003} to study the relation between the mean curvature of hypersurfaces and the volume of sections of balls (we reserve their notation $B^+_p(\vep)$ for the half-ball).

\begin{definition}
	Let $\cS$ be a smooth hypersurface in $\RR^{n+1}$ with a locally chosen normal vector field $\bN:\cS\rightarrow\RR^{n+1}$. Let $B^{(n+1)}_p(\vep)$ be a ball of radius $\vep>0$ centered at a point $p\in\cS$, for small enough $\vep$ the hypersurface always separates this ball into two connected components. Consider the region $V^+_p(\vep)$ to be that spherical component such that $\bN(p)$ points towards inside the region $V^+_p(\vep)$.
\end{definition}

All the methods and results of \cite{pottmann2007} for surfaces using this domain generalize because to approximate integrals of functions over this type of region in $\RR^3$, the formula developed in their work makes use of the hypersurface approximations of \cite{hulin2003}, valid in any dimension.

\begin{lemma}\label{lem:intApprox}
	Let $f:\RR^{n+1}\rightarrow\RR$ be a function of order $\cO(\rho^k z^l)$ in cylindrical coordinates $\bs{X}=(\bx,z)=(\rho\bs{\cx},z),\;\bs{\cx}\in\SS^{n-1}$, let $\cS$ be a graph hypersurface given by the function $z(\bx)$ whose normal at the origin points in the positive $z$-axis, and $V^+_p(\vep)$ the spherical component delimited by this $\cS$, then
\begin{equation}
	\int_{V^+_p(\vep)} f(\bs{X})\text{\emph{dVol}} = \int_{B^+_p(\vep)}f(\bs{X})\text{\emph{dVol}}\; -\; \int_{B^{n}_p(\vep)}\left[ \int_{z=0}^{z=\frac{1}{2}\sum_{\mu=1}^n\ka_\mu x^2_\mu } f(\bx,z)\, dz\right]\, d^n\bx + \cO(\vep^{k+2l+n+3})
\end{equation}
	where the half ball $B^+_p(\vep)$ consists of the points of $B^{n+1}_p(\vep)$ such that $z\geq 0$.
\end{lemma}
\begin{proof}
We are going to approximate $z(\bx)$ by its osculating paraboloid at the origin $\frac{1}{2}\sum_{\mu=1}^n\ka_\mu x^2_\mu$, and remove from the complete half-ball integral of $f(\bs{X})$ its contribution from below the paraboloid but, since what we know is the function over the tangent space at $p$, what can be computed is the contribution below the paraboloid over a domain in the tangent space that is explicitly integrable. The exact domain is determined by the sphere intersection with the hypersurface, $\{\|\bx\|^2 + z(\bx)^2\leq\vep^2\}$, and what we can compute exactly is the integral over the cylinder $\{\rho\leq\vep\}$, so that for every $\bx\in B^{(n)}_p(\vep)\subset T_p\cS$, we can remove the contribution of $\int_0^z f(\bx,z)dz$. This results in the approximation:
$$
\int_{V^+_p(\vep)} f(\bs{X})\text{dVol} \approx \int_{B^+_p(\vep)}f(\bs{X})\text{dVol}\; -\; \int_{B^{n}_p(\vep)}\left[ \int_{z=0}^{z(\bx)} f(\bx,z(\bx))\, dz\right]\, d^n\bx .
$$
What we need to find is the order of the error in this expression. The volume in the second integral extends outside the ball that defines $V^+_p(\vep)$, which is inscribed in the cylinder, and thus the integral below the hypersurface is subtracting an extra contribution from the region $\Omega$, that lies outside the sphere but inside the cylinder and is bounded by the hypersurface. Then
$$
\int_\Omega f(\bs{X})\text{dVol}\;\; \leq\;\; \max_{\bs{X}\in\Omega}|f(\bs{X})|\cdot \text{Vol}(\Omega).
$$
Since $z(\rho\bs{\cx})\sim\cO(\rho^2)$, we have $\max_{\bs{X}\in\Omega}|f(\bs{X})|\sim\cO(\rho^k(\rho^2)^l)$. To bound the volume of $\Omega$, notice $\rho$ is bounded by $\vep$ from the cylinder and by approximately $\vep-C\vep^3$ from the intersection of the sphere with the hypersurface, for some constant $C$ (cf. lemma \ref{lem:boundary} below or the estimation in \cite{hulin2003}). This maximum thickness $\cO(\vep^3)$ is added up for every point of the base sphere, whose area is $\sim\cO(\vep^{n-1})$. Now, the maximum height in the $z$ direction of $\Omega$ is of order $\cO(\vep^2)$ because it is given by the intersection of the cylinder with the hypersurface. Therefore, $\text{Vol}(\Omega)\sim\cO(\vep^2\vep^{n-1}\vep^3)\sim\cO(\vep^{n+4})$. The total error of this approximation is then $\cO(\vep^{k+2l+n+4})$
Finally, the graph function $z(\bx)$ is to be approximated by its osculating paraboloid, truncating the terms $\cO(\rho^3)$ from its Taylor series. This makes a new error in the second integral of our formula, given by the integral over the region inbetween the paraboloid and the actual hypersurface, which has height given by the $\cO(\rho^3)$ difference between the full series of $z$ and the quadratic terms. Therefore, the integral we are neglecting by this truncation makes an error
$$
\int_{\SS^{n-1}}\;\int_{\rho=0}^{\rho=\vep}\cO(\rho^k(\rho^2)^l)\cO(\rho^3)\rho^{n-1}d\rho\,d\,\SS\sim \cO(\vep^{k+2l+n+3})
$$
which is the leading order of the two errors for small $\vep>0$.
\end{proof}

The key idea of the approximations carried out in the previous lemma were developed by \cite{hulin2003} precisely to obtain the first integral invariant.

\begin{proposition}[Hulin and Troyanov]
	The volume of the spherical component cut by a hypersurface has the asymptotic expansion
	\begin{equation}\label{eq:volSph}
	V(V^+_p(\vep))=\frac{V_{n+1}(\vep)}{2} - \frac{\vep^2\,V_n(\vep)}{2(n+2)}H + \cO(\vep^{n+3}).
	\end{equation} 
\end{proposition}
\begin{proof}
	Using lemma \ref{lem:intApprox} the computation is immediate since $\int_{B^+_p(\vep)}\text{dVol}=\frac{V_{n+1}(\vep)}{2}$, and
$$
\int_{B^{n}_p(\vep)}\left[ \int_{z=0}^{z=\frac{1}{2}\sum_{\mu=1}^n\ka_\mu x^2_\mu } dz\right]\, d^n\bx =\frac{1}{2}\sum_{\mu=1}^n\ka_\mu\left[\int_{B^{n}_p(\vep)} x^2_\mu\, d^n\bx\right] =\frac{D_2}{2}\sum_{\mu=1}^n\ka_\mu .
$$\end{proof}

\begin{proposition}
	The barycenter of the spherical component is of the form:
\begin{equation}
	\bb(V^+_p(\vep)) = [\;0,\dots,0,\; 2\frac{V_n(\vep)}{V_{n+1}(\vep)}\frac{\vep^2}{n+2}\left( 1+ \frac{V_n(\vep)}{V_{n+1}(\vep)}\frac{\vep^2}{n+2} H \right) \; ]^T \; +\cO(\vep^3).
\end{equation}
\end{proposition}
\begin{proof}
Notice that $\int_{V^+_p(\vep)}\bx\,\text{dVol} = \cO(\vep^{n+4})$ because applying lemma \ref{lem:intApprox}, $\int_{B^+_p(\vep)}\bx\,d^n\bx\,dz=0$, and the second integral is also of monomials of odd degree. We get right away the normal component
$$
	[V(V^+_p(\vep))\bb(V^+_p(\vep))]_z = \int_{B^+_p(\vep)}z\,d^n\bx\,dz -\int_{B^n_p(\vep)}\frac{1}{2}\left[ \sum_{\mu=1}^n\ka_\mu x^2_\mu \right]^2\,d^n\bx+\cO(\vep^{n+4}) = D^{(n+1)}_1+\cO(\vep^{n+4})
$$
where we have discarded the second integral since its order is $\cO(D^{(n)}_4)=\cO(D^{(n)}_{22})\sim\cO(\vep^{n+4})$, which leaves the same order $\cO(\vep^3)$ as the error after dividing by the volume. The final expression follows from inverting the volume formula from the previous proposition and using the value of $D_1^{(n+1)}$ from the appendix.
\end{proof}

\begin{theorem}\label{th:sphComp}
	The covariance matrix $C(V^+_p(\vep))$ has eigenvalues with the following series expansion, for all $\mu=1,\dots,n$:
\begin{align}\label{eq:eigenSphComp}
	& \lbd_{\mu}(V^+_p(\vep)) = V_{n+1}(\vep)\frac{\vep^2}{2(n+3)}-V_n(\vep)\frac{\vep^4}{2(n+2)(n+4)}(2\ka_\mu +H) +\cO(\vep^{n+5}), \\[3mm]
	& \lbd_{n+1}(V^+_p(\vep)) = V_{n+1}(\vep)\frac{\vep^2}{2(n+3)}- 2\frac{V_n(\vep)^2}{V_{n+1}(\vep)}\frac{\vep^4}{(n+2)^2}\left(1+ \frac{V_n(\vep)}{V_{n+1}(\vep)}\frac{\vep^2}{n+2} H\right) + \cO(\vep^{n+5}).	
\end{align}
Moreover, in the limit $\vep\rightarrow 0^+$, when the principal curvatures are different, the corresponding eigenvectors $\bE_\mu(V^+_p(\vep))$ converge linearly to the principal directions of $\cS$ at $p$, and $\bE_{n+1}(V^+_p(\vep))$ converges quadratically to the hypersurface normal vector $\bN$ at $p$.
\end{theorem}
\begin{proof}
Working in the basis formed by the principal directions and the normal vector of the hypersurface at the fixed point $p$, we shall compute the entries of the covariance matrix and see that it is diagonal to all orders smaller than $\cO(\vep^{n+5})$, precisely the error we get in the diagonal elements, therefore the eigenvalues coincide with those diagonal terms up to that error since differences between eigenvalues of symmetric matrices are bounded by the matrix norm distance. 
The covariance matrix splits into the first two terms of
$$
C(V^+_p(\vep)) =\int_{V^+_p(\vep)} \bs{X}\otimes\bs{X}^T\;\text{dVol} -\int_{V^+_p(\vep)}\bs{X}\otimes\bb^T\;\text{dVol} -\int_{V^+_p(\vep)} \bb\otimes\bs{X}^T\;\text{dVol} + \int_{V^+_p(\vep)} \bb\otimes\bb^T\;\text{dVol},
$$
because the last three terms become the same upon integration. To compute the term left we can use the expression for $V\bb$ from the proof of the barycenter formula to get:
\begin{align*}
	\int_{V^+_p(\vep)} \bb\otimes\bb^T\;\text{dVol} = V(V^+_p(\vep))\bb \otimes\bb^T = 
\left( \begin{array}{@{}c|c@{}}
   \begin{matrix}
       &  &   \\
       & \cO(\vep^{n+7})_{n\times n} &  \\
       &  &  
   \end{matrix} 
      & \cO(\vep^{n+5})_{n\times 1} \\
   \cmidrule[0.4pt]{1-2}
   \cO(\vep^{n+5})_{1\times n} & V(V^+_p(\vep))s^2_z \\
\end{array} \right)
\end{align*}
where
$$
V(V^+_p(\vep))s^2_z = \frac{[D_1^{(n+1)}]^2}{V(V^+_p(\vep))} + \cO(\vep^{n+5}).
$$
The other only contribution to the last entry of the covariance matrix is
$$
\int_{V^+_p(\vep)} z^2\;\text{dVol} = \int_{B^+_p(\vep)} z^2\, d^n\bx\, dz -\frac{1}{24}\int_{B^n_p(\vep)}\left[ \sum_{\mu=1}^n\ka_\mu x^2_\mu \right]^3\,d^n\bx +\cO(\vep^{n+7}) = \frac{D^{(n+1)}_2}{2}+\cO(\vep^{n+6}),
$$
in which we have neglected the second integral for being of higher order than the barycenter matrix error, whose subtraction yields the stated result for the normal eigenvalue. Notice that the other elements in the last column and row of the complete covariance matrix are $\cO(\vep^{n+5})$ since the remaining contributions come from $\int_{V^+_p(\vep)}x_\mu z\; \text{dVol} \sim\cO(\vep^{n+6})$, and its approximation formula has all monomials with odd powers in $x$.

Now, we compute the tangent coordinates block. This can be done at once for any $\mu,\nu=1,\dots,n$, noticing that when $\mu\neq\nu$, the integrals of lemma \ref{lem:intApprox} are of monomials of odd degree in tangent coordinates so the off-diagonal elements are $\cO(\vep^{n+5})$:
\begin{align*}
& \int_{V^+_p(\vep)} x_\mu^2 \;\text{dVol}  = \int_{B^+_p(\vep)}x^2_\mu\, d^n\bx\,dz - \int_{B^{n}_p(\vep)}x_\mu^2\left(\frac{1}{2}\sum_{\al=1}^n\ka_\al x_\al^2\right) \, d^n\bx + \cO(\vep^{n+5}) \\
& = \frac{D^{(n+1)}_2}{2}-  \frac{1}{2}\int_{B^{n}_p(\vep)}\left(\ka_\mu x^4_\mu+\sum_{\al\neq\mu}\ka_\al x^2_\al x^2_\mu\right)+  \cO(\vep^{n+5}) \\
& = \frac{D^{(n+1)}_2}{2}-\frac{D^{(n)}_4}{2}\ka_\mu-\frac{D^{(n)}_{22}}{2}\sum_{\al\neq\mu}\ka_\al + \cO(\vep^{n+5}) =\frac{D^{(n+1)}_2}{2} -\frac{D^{(n)}_{22}}{2}(2\ka_\mu + H) +  \cO(\vep^{n+5})
\end{align*}
Here we have completed the last sum and used the fact that $D_4=3D_{22}$.

The perturbation theory of Hermitian matrices \cite{davis1970}, \cite{kato1982} shows the convergence of the eigenvectors to the principal directions in the case of no multiplicity: truncating $C(V^+_p(\vep))$ to order lower than $\cO(\vep^{n+5})$, that is precisely the order of the perturbation with respect to the exact diagonalized matrix. Fixing an eigenvalue $\lbd_\mu(V^+_p(\vep))$ with $\mu\neq n+1$, the minimum difference to the other eigenvalues is of order $\sim \vep^{n+4}(\ka_\mu -\ka_\nu)$, whereas for the last eigenvalue its distance to all the others is already at leading order $\sim \vep^{n+3}$. Therefore, from the $\sin\theta$ theorem \cite{davis1970}, the perturbation $\cO(\vep^{n+5})$ changes the eigenvectors $\{\bE_\mu(V^+_p(\vep))\}_{\mu=1}^n$ with respect to the principal directions as $\cO(\vep^{n+5})/\cO(\vep^{n+4}(\ka_\mu -\ka_\nu))\sim\frac{\vep}{\ka_\mu -\ka_\nu}$, and changes the eigenvector $\bE_{n+1}(V^+_p(\vep))$ with respect to the normal as $\cO(\vep^{n+5})/\cO(\vep^{n+3})\sim\vep^2$, i.e., in the limit $\vep\rightarrow 0^+$ the eigevectors of $C(V^+_p(\vep))$ get a vanishing correction with respect to the principal and normal directions. 
\end{proof}

Therefore, since the Weingarten operator $\hat{\bs{S}}$ at $p$ is $\diag(\ka_1(p),\dots,\ka_n(p))$ in our basis, we may write the covariance matrix as:
\begin{align*}
	C(V^+_p(\vep)) = \frac{V_{n+1}(\vep)\,\vep^2}{2(n+3)}\,\Id_{n+1}\; -\;\frac{V_n(\vep)\,\vep^4}{(n+2)(n+4)}
\left( \begin{array}{@{}c|c@{}}
   \begin{matrix}
       &  &   \\
       & \bs{\hat{S}} + \frac{H}{2}\Id_{n} &  \\
       &  &  
   \end{matrix} 
      & 0_{n\times 1} \\
   \cmidrule[0.4pt]{1-2}
   0_{1\times n} & 2\frac{V_n(\vep)(n+4)}{V_{n+1}(\vep)(n+2)} \\
\end{array} \right) \; +\; \cO(\vep^{n+5}).
\end{align*}

In \cite{pottmann2007}, the spherical shell $V^+_p(\vep)\cap \SS^{n}_p(\vep)$ is also considered for surfaces in $\RR^3$ following \cite{connolly1986}, and its invariants are shown to be just the derivative with respect to scale of those obtained for the ball region. This is due to the fact that the integral of a function over a region delimited by a ball is the radial integration of the corresponding result over spheres. The same property holds in our case, therefore the derivatives with respect to $\vep$ of the invariants in this section are the corresponding integral invariants of the $n$-dimensional spherical shell.

%%%%%%%%%%%%%%%%%%%%%%%%%%%%%%%%%%%%%%%%%%%%%%%%%%%%%%%%%%%%%%%%%%%%%%
%%%%%%%%%%%%%%%%%%%%%%%%%%%%%%%%%%%%%%%%%%%%%%%%%%%%%%%%%%%%%%%%%%%%%%
%%%%%%%%%%%%%%%			4. PATCH
%%%%%%%%%%%%%%%%%%%%%%%%%%%%%%%%%%%%%%%%%%%%%%%%%%%%%%%%%%%%%%%%%%%%%%
%%%%%%%%%%%%%%%%%%%%%%%%%%%%%%%%%%%%%%%%%%%%%%%%%%%%%%%%%%%%%%%%%%%%%%

\section{Hypersurface Patch Integral Invariants}\label{sec:sphPatch}

Now, we shall compute the asymptotic expansions of the integral invariants of the hypersurface patch cut out by a ball centered at $p$ and radius $\vep >0$, i.e. over the domain $D_p(\vep)=\cS\cap B_p^{n+1}(\vep)$, using again the local graph representation in a small neighborhood around the point.

Since a parametrization of the region is needed to perform the integrals locally, we need to find local parametric equations of the boundary $\dd(\cS\cap B^{n+1}_p(\vep))$ to high enough order in $\vep$ so that we can expand asymptotically the integral invariants in terms of the geometric information of the hypersurface at the point. The strategy of \cite{pottmann2007}, hinted in \cite{hulin2003}, obtaining a cylindrical coordinate approximation for the boundary radius of the patch, works in general dimension as follows.

\begin{lemma}\label{lem:boundary}
	In cylindrical coordinates $(\rho,\phi_1,\dots,\phi_{n-1}, z)$ over the tangent space $T_p\cS$, fixing the basis to the principal directions and the normal vector of $\cS$ at $p$, the parametric equations of a point $\bs{X}=(\rho\cx_1,\dots,\rho\cx_n, z)^T$ in $\dd D_p(\vep)=\cS\cap \SS^{n}_p(\vep)$, are
\begin{equation}
	r(\bs{\cx}):=\rho(\cx_1,\dots,\cx_n) = \vep -\frac{1}{8}\ka^2(\bs{\cx})\vep^3 +\cO(\vep^4) ,\quad\quad z(\cx_1,\dots,\cx_n)=\frac{1}{2}\ka^2(\bs{\cx})\vep^2 +\cO(\vep^3),
\end{equation}
	where $\cx_1,\dots,\cx_n$ are the coordinates of points on $\SS^{n-1}\subset T_p\cS$, and
\begin{equation}
	\ka(\bs{\cx})=\ka(\cx_1,\dots,\cx_n) = \sum_{\mu=1}^n \ka_\mu\cx^2_\mu
\end{equation}
is the normal curvature of $\cS$ at $p$ in the direction of $\bs{\cx}$.
\end{lemma}
\begin{proof}
	In this cylindrical coordinate system the expansion of the function that locally defines $\cS$ is
	$$
	z(\bx)=\frac{1}{2}\sum_{\mu=1}^n\ka_\mu x^2_\mu + \cO(x^3) =\frac{1}{2}\ka(\bs{\cx}) \rho^2 + \cO(\rho^3)
	$$ 
	since $x_\mu=\rho\cx_\mu$, and because $\II$ is diagonal in our basis with $\bs{\cx}$ a unit vector, equation \ref{eq:normalCurv} yields $\sum_{\mu=1}^n \ka_\mu\cx^2_\mu=\ka(\bs{\cx})$. Now, a point $\bs{X}=(\rho\cx_1,\dots,\rho\cx_n, z)^T$ in $\cS\cap \SS^{n}_p(\vep)$ satisfies the equation of the sphere $\|\bs{X}\| =\vep$, which in these coordinates is $\rho^2 + z^2 =\vep^2$.
	Substituting the expansion of $z(\bx)$ above, the resulting identity is
	$$
	\frac{1}{4}\ka(\bs{\cx})^2 \rho^4 + \rho^2 -\vep^2 + \cO(\rho^5) =0
	$$
	which up to order $4$ is a biquadratric equation in $\rho$ whose positive solution is
	$$
	\rho^2 = \frac{2}{\ka(\bs{\cx})^2}\left(-1 +\sqrt{1+\ka(\bs{\cx})^2\vep^2}\right) = \vep^2 -\frac{1}{4}\ka(\bs{\cx})^2\vep^4+\cO(\vep^6),
	$$
then taking roots once more the desired result obtains.
\end{proof}

For this type of domain the previous parametric expansions are enough to asymptotically expand both the integrand and the measure, collect terms and split integrals into those of the appendix \ref{sec:appendix}.

The area or mass of the domain can be expressed as a correction to the volume of the $n$-ball in terms of the extrinsic and intrinsic curvature of $\cS$ at the point. This compares with the case of the volume of a geodesic ball domain inside a manifold \cite{gray1979}, which exhibits a correction only dependent on the intrinsic scalar curvature $\cR$.

\begin{proposition}
	The $n$-dimensional area of the hypersurface patch has the asymptotic expansion
	\begin{equation}\label{eq:volPatch}
		V(D_p(\vep)) = V_n(\vep)\left[1+ \frac{\vep^2}{8(n+2)}(H^2-2\cR)  + \cO(\vep^3)\right].
	\end{equation}
\end{proposition}
\begin{proof}
	Using lemma \ref{lem:boundary} in equation \ref{eq:volElem}, we have that
$$
\text{dVol}\vert_{D_p(\vep)}=\sqrt{\det g(\bx)}\, dx_1\cdots dx_n = \left[1+\frac{1}{2}\sum_{\mu=1}^n\ka^2_\mu x^2_\mu + \cO(x^3)\right]dx_1\cdots dx_n,
$$
since $\sum_{\mu=1}^n\left(\frac{\dd z}{\dd x_\mu}\right)^2=\|\nabla z (\bx)\|^2$ can be considered small for small enough $\vep >0$, because in our coordinates $\nabla z(\bze)=0$. With this and the cylindrical measure, eq. \ref{eq:cylMeasure}, the integration becomes
\begin{align*}
& V(D_p(\vep)) = \int_{\cS\cap B^{n+1}_p(\vep)}\text{dVol}=\int_{\SS^{n-1}}d\,\SS\int_0^{r(\bs{\cx})}\left[ 1+\frac{1}{2}\sum_{\mu=1}^n\ka^2_\mu\rho^2\cx^2_\mu +\cO(\rho^3)\right]\rho^{n-1}\, d\rho \\
& = \int_{\SS^{n-1}}d\,\SS\left[ \frac{1}{n}( \vep -\frac{\ka(\bs{\cx})^2\vep^3}{8} +\cO(\vep^4))^n + \frac{1}{2}\sum_{\mu=1}^n\ka^2_\mu\cx^2_\mu\frac{1}{n+2}( \vep -\frac{\ka(\bs{\cx})^2\vep^3}{8} +\cO(\vep^4) )^{n+2} +\cO(\vep^{n+4})\right]
\end{align*}
after integrating over $\rho$ up to the boundary radius. Expanding the binomial series and the square of the normal curvature, all the remaining integrals are in example \ref{ex:constants}, leading to
\begin{align*}
 V(D_p(\vep)) & = \frac{\vep^n}{n}S_{n-1}-\frac{\vep^{n+2}}{8}\int_{\SS^{n-1}}\ka(\bs{\cx})^2\, d\,\SS + \frac{\vep^{n+2}}{2(n+2)}\sum_{\mu=1}^n\ka^2_\mu \int_{\SS^{n-1}}\cx^2_\mu\, d\,\SS + \cO(\vep^{n+3}) \\
& = V_n(\vep)-\frac{\vep^{n+2}}{8}\int_{\SS^{n-1}}d\,\SS\left( \sum_{\mu=1}^n\ka^2_\mu \cx^4_\mu +2 \sum_{\mu<\nu}^n\ka_\mu\ka_\nu\cx^2_\mu\cx^2_\nu \right) + \frac{C_2\,\vep^{n+2}}{2(n+2)}\sum_{\mu=1}^n\ka^2_\mu + \cO(\vep^{n+3}) \\
& = V_n(\vep) + \frac{\vep^{n+2}}{n+2}\left[ \left( \frac{C_2}{2}-\frac{n+2}{8}C_4 \right)\sum_{\mu=1}^n\ka^2_\mu - C_{22}\frac{n+2}{8}2 \sum_{\mu<\nu}^n\ka_\mu\ka_\nu \right]+ \cO(\vep^{n+3}) \\
& = V_n(\vep) + \frac{\vep^{n+2}}{n+2}\left[\frac{C_2}{8}(H^2-\cR)-\frac{C_2}{8}\cR \right] + \cO(\vep^{n+3}), 
\end{align*}
where we use equation \ref{eq:trS2} and the relations among the coefficients from the appendix.
\end{proof}

It is natural to expect the extrinsic curvature $H$ to be present in the second order correction since the domain is determined by a ball in ambient space, hence it depends on how $\cS$ is embedded, in contrast to an intrinsically defined geodesic ball where the correction only depends on $\cR$.

The center of mass in this case turns out to deviate, to leading order in $\vep$, only in the normal direction with respect to the center of the ball.

\begin{proposition}\label{prop:patchBary}
	The barycenter of the patch region has coordinates in the principal basis with respect to $p$ given by
\begin{equation}\label{eq:patchBary}
	\bb(D_p(\vep)) = [\,\cO(\vep^4),\dots, \cO(\vep^4),\, \frac{\vep^2}{2(n+2)}H +\cO(\vep^3)\,]^T.
\end{equation}
\end{proposition}
\begin{proof}
When integrating any tangent component $x_\al$ of $\bs{X}$, only factors with an odd power in some component are produced because the known terms (see previous proof) now contain products $\cx_\al\cx^2_\mu,\;\cx_\al\cx_\mu^4$ and $\cx_\al\cx_\mu^2\cx_\nu^2$, which always have an odd power factor regardless of the subindices combination. Therefore the first $n$ components of $V(D_p(\vep))\bb(D_p(\vep))$ are of order $\cO(\vep^{n+4})$, coming from the error inside $r(\bs{\cx})^{n+1}$ after integrating radially the first term $x_\al\rho^{n-1}d\rho$.
The normal component of $\bs{X}$ integrates as 
\begin{align*}
	\int_{\cS\cap B^{n+1}_p(\vep)}z\;\text{dVol} & =\int_{\SS^{n-1}}d\,\SS\int_0^{r(\bs{\cx})}\left[ \frac{1}{2}\ka(\bs{\cx})\rho^2+\cO(\rho^3) \right]\left[ 1+\frac{1}{2}\sum_{\mu=1}^n\ka^2_\mu\rho^2\cx^2_\mu +\cO(\rho^3)\right]\rho^{n-1}\, d\rho \\
& = \int_{\SS^{n-1}}d\,\SS\left[ \frac{\ka(\bs{\cx})}{2(n+2)}(\vep -\frac{\ka(\bs{\cx})^2}{8}\vep^3 +\cO(\vep^4))^{n+2} + \cO(\vep^{n+3}) \right] \\
& = \frac{\vep^{n+2}}{2(n+2)}\sum_{\mu=1}^n\ka_\mu\int_{\SS^{n-1}}\cx^2_\mu\,d\,\SS + \cO(\vep^{n+3}) = C_2\frac{\vep^{n+2}}{2(n+2)}H+ \cO(\vep^{n+3}).
\end{align*}
Then normalizing by the volume to lowest order cancels the coefficient $C_2\vep^{n}$.
\end{proof}

Finally, the study of the covariance matrix of the patch domain shows a behavior similar to the spherical component, but where the next-to-leading order contribution to the eigenvalues includes only products of principal curvatures and no linear terms.

\begin{theorem}
	The covariance matrix $C(D_p(\vep))$ has $n$ eigenvalues that scale like $\vep^{n+2}$ as
	\begin{equation}\label{eq:eigen1Patch}
		\lambda_\mu(D_p(\vep)) = V_n(\vep)\left[ \frac{\vep^2}{n+2} + \frac{\vep^4}{8(n+2)(n+4)}(H^2-2\cR -4H\ka_\mu)\right] +\cO(\vep^{n+5}),
	\end{equation}
	for all $\mu=1,\dots,n$, and one eigenvalue scaling as $\vep^{n+4}$ with leading term
	\begin{equation}\label{eq:eigen2Patch}
		\lambda_{n+1}(D_p(\vep)) = V_n(\vep)\frac{\vep^4}{2(n+2)(n+4)}\left(\frac{n+1}{n+2}H^2 - \cR \right) +\cO(\vep^{n+5}).	
	\end{equation}
	Moreover, in the limit $\vep\rightarrow 0^+$, if the principal curvatures at $p$ are all different, the  eigenvectors $\bE_\mu(D_p(\vep))$ corresponding to the first $n$ eigenvalues converge  to the principal directions of $\cS$ at $p$, and the last eigenvector $\bE_{n+1}(D_p(\vep))$ converges to the hypersurface normal vector $\bN(p)$.
\end{theorem}
\begin{proof}
We need to evaluate $\int_{D_p(\vep)} \bs{X}(\bx)\otimes \bs{X}(\bx)^T\;\sqrt{\det g}\, d^n\bx$ and $V(D_p(\vep))\bb(D_p(\vep))\otimes\bb(D_p(\vep))^T$. The latter can be obtained from the previous proof:
$$
[\,\cO(\vep^{n+4}),\dots, \cO(\vep^{n+4}),\, \frac{C_2\vep^{n+2}}{2(n+2)}H +\cO(\vep^{n+3})\,]^T\otimes [\,\cO(\vep^4),\dots, \cO(\vep^4),\, \frac{\vep^2}{2(n+2)}H +\cO(\vep^3)\,],
$$
resulting in all entries of the $n\times n$ block being $\cO(\vep^{n+8})$, the first $n$ elements of the last column and last row being $\cO(\vep^{n+6})$, and the last element of the matrix becoming
$$
 [V(D_p(\vep))\bb(D_p(\vep))\otimes\bb(D_p(\vep))^T]_{(n+1),(n+1)}=\frac{V_n(\vep)\,\vep^4}{4(n+2)^2}H^2 +\cO(\vep^{n+5}),
$$
(we already disregarded the term of $\cO(\vep^{n+6})$ that can be computed for this matrix entry because we shall see below that the other contributing term in that position has error at $\cO(\vep^{n+5})$).

Now, the rest of the covariance matrix requires the longest computations so far. The entries of $\bs{X}(\bx)\otimes \bs{X}(\bx)^T$ are of three types: $x_\mu x_\nu$, $x_\mu z(\bx)$ and $z(\bx)^2$. The first $n$ entries of the last column and last row, $x_\mu z(\bx)$, contribute at order $\cO(\vep^{n+4})$. This implies that the matrix may not decompose at order $\cO(\vep^{n+4})$ as direct sum of a "tangent" $n\times n$ block, the integrals of $[x_\mu x_\nu]$, and a "normal" $1\times 1$ block, the integral of $z(\bx)^2$. However, this will not affect the eigenvalue decomposition as we shall see below, and the eigenvalues will be given to order $\cO(\vep^{n+4})$ by the diagonals of these blocks.

The normal block entry is:
\begin{align*}
\int_{\cS\cap B^{n+1}_p(\vep)}z^2\;\text{dVol} & = \int_{\SS^{n-1}}d\,\SS\int_0^{r(\bs{\cx})}\left[ \frac{1}{4}\ka(\bs{\cx})^2\rho^4+\cO(\rho^5) \right]\left[ 1+\frac{1}{2}\sum_{\mu=1}^n\ka^2_\mu\rho^2\cx^2_\mu +\cO(\rho^3)\right]\rho^{n-1}\, d\rho \\
& = \int_{\SS^{n-1}}d\,\SS\int_0^{r(\bs{\cx})}\left[ \frac{1}{4}\sum_{\al,\bet=1}^n\ka_\al\ka_\bet\cx^2_\al\cx^2_\bet\,\rho^{n+3} + \cO(\rho^{n+4})\right]d\rho \\
& = \frac{1}{4}\left[\sum_{\al=1}^n\ka_\al^2\int_{\SS^{n-1}}\cx_\al^4\,d\,\SS +2\sum_{\al <\bet}^n\ka_\al\ka_\bet\int_{\SS^{n-1}}\cx_\al^2\cx^2_\bet\,d\,\SS\right]\frac{\vep^{n+4}}{n+4} +\cO(\vep^{n+5}) \\
& = \frac{\vep^{n+4}}{4(n+4)}\left[C_4(H^2-\cR) +C_{22}\cR\right] +\cO(\vep^{n+5});
\end{align*}
subtracting the contribution from the barycenter matrix term, the last eigenvalue becomes
$$
\lambda_{n+1}(p,\vep)=\frac{C_2\,\vep^{n+4}}{4(n+2)(n+4)}\left[3H^2-2\cR\right] - \frac{C_2\,\vep^{n+4}}{4(n+2)^2}H^2 +\cO(\vep^{n+5})
$$
which simplifies to the stated result.

The tangent block entries can be computed simultaneously considering arbitrary $\mu,\nu=1,\dots, n$:
\begin{align*}
& \int_{\cS\cap B^{n+1}_p(\vep)}x_\mu x_\nu\;\text{dVol} = \int_{\SS^{n-1}}d\,\SS\int_0^{r(\bs{\cx})}\rho^2\cx_\mu\cx_\nu\rho^{n-1}\left[ 1+\frac{1}{2}\sum_{\al=1}^n\ka^2_\al\rho^2\cx^2_\al +\cO(\rho^3)\right]\, d\rho \\
& =  \int_{\SS^{n-1}}\; d\,\SS \frac{\cx_\mu\cx_\nu}{n+2}(\vep-\frac{\ka(\bs{\cx})^2}{8}\vep^{3}+\cO(\vep^4))^{n+2} +  \frac{1}{2}\sum_{\al=1}^n\ka^2_\al\int_{\SS^{n-1}}\cx^2_\al\cx_\mu\cx_\nu\,d\,\SS\frac{\vep^{n+4}}{n+4}+\cO(\vep^{n+5}) \\
& = \frac{\vep^{n+2}}{n+2}\left[ \del_{\mu\nu}C_2-\frac{\vep^2(n+2)}{8}\int_{\SS^{n-1}}d\,\SS\,\cx_\mu\cx_\nu\left(\sum_{\al=1}^n\ka^2_\al\cx^4_\al + 2\sum_{\al <\bet}^n\ka_\al\ka_\bet\cx^2_\al\cx^2_\bet \right) \right] \; + \\
& \quad\quad\quad\quad\quad\quad\quad\quad\quad\quad + \frac{\vep^{n+4}}{n+4}\frac{\del_{\mu\nu}}{2}\left(\ka^2_\mu\int_{\SS^{n-1}}\cx^4_\mu\, d\,\SS +\sum_{\al\neq\mu}^n\ka^2_\al\int_{\SS^{n-1}}\cx^2_\al\cx^2_\mu\, d\,\SS \right) + \cO(\vep^{n+5}),
\end{align*}
where the $\del_{\mu\nu}$ appears because the monomials get an odd power if $\mu\neq\nu$. Now, the different integrals inside the indexed sums result in different constants depending on the different monomials that the terms $\cx^2_\mu\cx^4_\al$ and $\cx^2_\mu\cx^2_\al\cx^2_\bet$ can combine into, thus 
\begin{align*}
& = C_2\frac{\vep^{n+2}}{n+2}\del_{\mu\nu}+\frac{\vep^{n+4}}{n+4}\del_{\mu\nu}\left[ -\frac{n+4}{8}\left( C_6\ka^2_\mu + C_{24}\sum^n_{\al\neq\mu}\ka^2_\al + \sum^n_{\substack{\al,\bet \\ \al\neq\bet}}\ka_\al\ka_\bet(\del_{\al\mu}+\del_{\bet\mu})C_{24} \; + \right.\right. \\
& \quad\quad\quad\quad\quad\quad\quad\quad\quad\quad\quad\quad\quad\quad\; + \left.\left. \sum_{\substack{\al\neq\bet \\ \al,\bet\neq\mu}}^n\ka_{\al}\ka_{\bet}C_{222}\right) + \frac{C_4}{2}\ka^2_\mu + \frac{C_{22}}{2}\sum^n_{\al\neq\mu}\ka^2_\al \right] + \cO(\vep^{n+5})
\end{align*}
\begin{align*}
& = C_2\frac{\vep^{n+2}}{n+2}\del_{\mu\nu}+\frac{\vep^{n+4}}{n+4}\del_{\mu\nu}\left[ (\frac{C_4}{2}-\frac{n+4}{8}C_6)\ka^2_\mu + (\frac{C_{22}}{2}-\frac{n+4}{8}C_{24})\sum_{\al\neq\mu}^n\ka^2_\al \right. \\ 
& \quad\quad\quad\quad\quad\quad\quad\quad\quad\quad\quad\quad\quad\quad\;\left. -\frac{n+4}{8}( 2C_{24}\sum_{\al\neq\mu}^n\ka_\mu\ka_\al + C_{222}\sum_{\substack{\al\neq\bet \\ \al,\bet\neq\mu}}^n\ka_{\al}\ka_{\bet} ) \right] + \cO(\vep^{n+5}).
\end{align*}
Notice that the summations in the last equation are all over indices that must be different from $\mu$, so we can add and subtract the corresponding missing terms to those sums as long as we subtract them in the correct place. Doing this, and using the crucial relationships between the constants from the appendix, each of the different terms under the big braces simplify to:
\begin{align*}
& (\frac{C_4}{2}-\frac{n+4}{8}C_6-\frac{C_{22}}{2}+\frac{n+4}{8}C_{24})\, \ka^2_\mu = -\frac{C_2}{2(n+2)} \ka^2_\mu , \\[3mm]
& (\frac{C_{22}}{2}-\frac{n+4}{8}C_{24})\sum_{\al=1}^n\ka^2_\al = \frac{C_2}{8(n+2)}(H^2-\cR), \\
& -\frac{n+4}{8}( (2C_{24} -2C_{222})\sum_{\al\neq\mu}^n\ka_\mu\ka_\al + 2C_{222}\sum_{\al<\bet}^n\ka_{\al}\ka_{\bet}) = -\frac{C_2}{2(n+2)}R_{\mu\mu} + \frac{C_2}{8(n+2)}\cR.
\end{align*}
Finally, these add up into the expression
\begin{align*}
\int_{D_p(\vep)}x_\mu x_\nu\;\text{dVol} & = \del_{\mu\nu}V_n(\vep)\left[ \frac{\vep^2}{n+2}+\frac{\vep^4}{8(n+2)(n+4)}(-4\ka^2_\mu -4R_{\mu\mu} + H^2-2\cR) \right] +\cO(\vep^{n+5}),
\end{align*}
and since $\ka^2_\mu+R_{\mu\mu}=\ka_\mu H$, from equation \ref{eq:iiiForm}, the stated formula for the tangent eigenvalues follows from the diagonal of this block. Therefore, we can write $C(D^+_p(\vep))=$
\begin{align*}
\frac{V_n(\vep)\,\vep^2}{n+2}
\left( \begin{array}{@{}c|c@{}}
   \begin{matrix}
       &  &   \\
       &  \Id_{n} &  \\
       &  &  
   \end{matrix} 
      & 0_{n\times 1} \\
   \cmidrule[0.4pt]{1-2}
   0_{1\times n} & 0  \\
\end{array}  \right)
	+\frac{V_n(\vep)\,\vep^4}{2(n+2)(n+4)}
\left( \begin{array}{@{}c|c@{}}
   \begin{matrix}
       &  &   \\
       &  \frac{H^2-2\cR}{4}\Id_{n} -H\,\bs{\hat{S}} &  \\
       &  &  
   \end{matrix} 
      & A_{n\times 1} \\
   \cmidrule[0.4pt]{1-2}
   A_{1\times n} &\frac{n+1}{n+2}H^2 - \cR  \\
\end{array} \right) + \cO(\vep^{n+5}),
\end{align*}
so the Weingarten operator appears inside the covariance matrix in this case as well but multiplied by the mean curvature, which is a term in equation \ref{eq:iiiForm}.

Notice however that the argument in the proof of theorem \ref{th:sphComp} to equate the diagonal elements of this expansion with that of the actual eigenvalues cannot be made here, since there are off-diagonal error elements at the same order as the diagonal approximation (the aforementioned contributions from $x_\mu\, z$). Nevertheless, we can show how these do not affect the eigenvalues at that order by writing the eigenvalue-eigenvector equation as a series expansion term by term, which is always possible and converges for Hermitian matrices of converging power series elements \cite{rellich1969}, such as our $C(D_p(\vep))$. The limit matrix when $\vep\rightarrow 0$ is the zero matrix $C(D_p(0))=0_{n+1\times n+1}$, with zero as eigenvalue, $\lbd(0)=\lbd^{(0)}$, of multiplicity $(n+1)$. Following \cite{rellich1969}, at positive $\vep>0$, this degeneracy branches out into $(n+1)$ eigenvalues $\lbd_i(\vep)\rightarrow\lbd^{(0)}$, and $(n+1)$ orthonormal eigenvectors $\bs{V}_i(\vep)$, hence, without loss of generality normalizing by $V_n(\vep)$, we can write for any $i$:
\begin{align*}
& [\; a\,\vep^2
\left( \begin{array}{@{}c|c@{}}
   \begin{matrix}
		\Id_{n}
   \end{matrix} 
      & 0_{n\times 1} \\
   \cmidrule[0.4pt]{1-2}
   0_{1\times n} & 0  \\
\end{array}  \right)
	+b\,\vep^4
\left( \begin{array}{@{}c|c@{}}
   \begin{matrix}
       A_{n\times n}
   \end{matrix} 
      & B_{n\times 1} \\
   \cmidrule[0.4pt]{1-2}
   B_{1\times n} & C  \\
\end{array} \right) + \cO(\vep^{5})\; ][\, \bs{V}^{(0)} + \bs{V}^{(1)}\vep+\bs{V}^{(2)}\vep^2+\dots \,] = \\
& = (\lbd^{(1)}\vep^1+\lbd^{(2)}\vep^2+\lbd^{(3)}\vep^3+\lbd^{(4)}\vep^4+\dots)[\, \bs{V}^{(0)} + \bs{V}^{(1)}\vep+\bs{V}^{(2)}\vep^2+\dots \,].
\end{align*}
Even though the limit is the zero matrix, its completely degenerate eigenvalue has any orthonormal basis of $\RR^{n+1}$ as system of eigenvectors; the discussion in \cite{rellich1969} shows that it is the perturbation of a matrix what determines the limit orthonormal eigenvectors, i.e., we cannot prescribe them ab initio.
The block $A_{n\times n}$ is diagonal by the computations so far. The eigenvalue difference between the diagonalized matrix and $C(D_p(\vep))$ is bounded by the matrix norm difference between them, as in the proof \ref{th:sphComp}, which now is of order $\cO(\vep^{4})$ from block $B$, implying $\lbd^{(1)}=\lbd^{(3)}=0$, and $\lbd^{(2)}_i=a$ for $i=1,\dots,n$, and $\lbd^{(2)}_i=0$ for $i=n+1$. Therefore, equating terms at $\cO(\vep^4)$ (which is possible because the matrix operations are elementwise combinations of basic operations of converging power series) we get:
\begin{align*}
& [\; a
\left( \begin{array}{@{}c|c@{}}
   \begin{matrix}
		\Id_{n}
   \end{matrix} 
      & 0_{n\times 1} \\
   \cmidrule[0.4pt]{1-2}
   0_{1\times n} & 0  \\
\end{array}  \right)
	-\lbd^{(2)}\,\Id_{n+1}\, ]\,\bs{V}^{(2)} = [\, \lbd^{(4)}\,\Id_{n+1}- b
\left( \begin{array}{@{}c|c@{}}
   \begin{matrix}
		A_{n\times n}
   \end{matrix} 
      & B_{n\times 1} \\
   \cmidrule[0.4pt]{1-2}
   B_{1\times n} & C  \\
\end{array}  \right) 
	 \, ]\,\bs{V}^{(0)}.
\end{align*}
If $\mu=1,\dots,n$, then $\lbd^{(2)}=a\neq 0$ in our case, so at second order
\begin{align*}
& [\; a
\left( \begin{array}{@{}c|c@{}}
   \begin{matrix}
		\Id_{n}
   \end{matrix} 
      & 0_{n\times 1} \\
   \cmidrule[0.4pt]{1-2}
   0_{1\times n} & 0  \\
\end{array}  \right)
	-\lbd^{(2)}_\mu\,\Id_{n+1}\, ]\,\bs{V}^{(0)} =
\left( \begin{array}{@{}c|c@{}}
   \begin{matrix}
		0_{n\times n}
   \end{matrix} 
      & 0_{n\times 1} \\
   \cmidrule[0.4pt]{1-2}
   0_{1\times n} & -a  \\
\end{array}  \right)\,\bs{V}^{(0)} = 0,
\end{align*}
which means $a[\bs{V}^{(0)}]_{n+1}=0$, so the normal component $\bs{V}^{(0)}_\perp=0$, for the first $n$ limit eigenvectors. Hence, looking back at the fourth order matrix equation, the first $n$ rows of the LHS are 0, and the terms $B_{n\times 1}$ do not contribute due to being multiplied by $[\bs{V}^{(0)}]_{n+1}\equiv 0$, therefore we obtain a block matrix equation determining the tangent components of the limit eigenvectors: $$0_{n\times 1}=[\lbd^{(4)}\,\Id_n -bA_{n\times n}]\,\bs{V}^{(0)}_\top ,$$ which is a diagonal eigenvalue-eigenvector equation, i.e., $\lbd^{(4)}_\mu=bA_{\mu\mu}$ for $\mu=1,\dots,n$, and so the first $n$ limit eigenvectors are the principal directions.
If $i=n+1$, a similar argument shows that the last eigenvector satisfies $\bs{V}^{(0)}_\top=0$, and $\lbd^{(4)}=bC$. Therefore, the our diagonal expansions of $C(D_p(\vep))$ are exactly the eigenvalue expansion at $\cO(\vep^{n+4})$.
\end{proof}

%%%%%%%%%%%%%%%%%%%%%%%%%%%%%%%%%%%%%%%%%%%%%%%%%%%%%%%%%%%%%%%%%%%%%%
%%%%%%%%%%%%%%%%%%%%%%%%%%%%%%%%%%%%%%%%%%%%%%%%%%%%%%%%%%%%%%%%%%%%%%
%%%%%%%%%%%%%%%			5. DESCRIPTORS
%%%%%%%%%%%%%%%%%%%%%%%%%%%%%%%%%%%%%%%%%%%%%%%%%%%%%%%%%%%%%%%%%%%%%%
%%%%%%%%%%%%%%%%%%%%%%%%%%%%%%%%%%%%%%%%%%%%%%%%%%%%%%%%%%%%%%%%%%%%%%

\section{Multi-scale Curvature Descriptors for Hypersurfaces}\label{sec:descrip}

By solving the second term in the expansion from our integral invariants, we can extract the curvature information they encode and write it in terms of the volume or eigenvalues at a fixed scale, where the latter can be computed without knowledge of the geometry of $\cS$. This means that these local statistical measurements of the underlying point set can be employed to reconstruct its differential geometry, e.g., from a discrete sample cloud of points. These estimators can be used in geometry processing to ignore details below a given scale and act as feature detectors at the chosen size.

From the results shown, the eigenvectors of the covariance matrix serve as estimators of the principal and normal directions of a hypersurface at any generic point, thus fulfilling as well the role of an adapted Galerkin basis in the sense of \cite{broomhead1987topological} and \cite{solis2000}.

Employing the asymptotic expressions of section \S\ref{sec:sphVol}, we invert the relations and solve for the principal curvatures.

\begin{corollary}
	Abbreviating the integral invariants of the spherical component as $\lbd_\mu(p,\vep)\equiv\lbd_\mu(V^+_p(\vep)), V_p(\vep)\equiv V(V^+_p(\vep))$, then the corresponding descriptors of the principal curvatures, at scale $\vep>0$ and point $p\in\cS$, are given by
\begin{equation}
	\ka_\mu(V^+_p(\vep)) = \frac{n+4}{\vep^4V_n(\vep)}\left[ \frac{\vep^2V_{n+1}(\vep)}{n+3}-(n+1)\lbd_\mu(p,\vep) +\sum_{\al\neq\mu}^n\lbd_\al(p,\vep) \right],
\end{equation}
or equivalently by
\begin{align}
	H(V^+_p(\vep)) & = \frac{(n+2)V_{n+1}(\vep)}{\vep^2V_n(\vep)}\left(1-2\frac{V_p(\vep)}{V_{n+1}(\vep)}\right), \\[3mm]
	\ka_\mu(V^+_p(\vep)) & = \frac{(n+2)(n+4)}{\vep^4V_n(\vep)}\left( \frac{\vep^2V_{n+1}(\vep)}{2(n+3)}-\lbd_\mu(p,\vep)\right)+ \frac{1}{2}H(V^+_p(\vep))  ,
\end{align}
with corresponding errors $|H(p)-H(V^+_p(\vep))|\leq\cO(\vep)$, and $|\ka_\mu(p)-\ka_\mu(V^+_p(\vep))|\leq\cO(\vep)$, for any $\mu=1,\dots,n$. The eigenvectors $\bE_\mu(V^+_p(\vep))$ and $\bE_{n+1}(V^+_p(\vep))$ are descriptors of the principal and normal directions respectively.
\end{corollary}
\begin{proof}
Let us define the coefficients at scale 
	$$a = \frac{\vep^2\, V_{n+1}(\vep)}{2(n+3)},\qquad b=-\frac{\vep^4 V_{n}(\vep)}{2(n+2)(n+4)},$$
then the tangent eigenvalues from equation \ref{eq:eigenSphComp} solve the principal curvatures
$$\ka_\mu = \frac{\lbd_\mu - a}{2b} -\frac{1}{2}H +\cO(\vep).$$
Fixing one $\mu=1,\dots,n$, and subtracting any two such equations with $\mu\neq\al$ results in
$$
\ka_\al =\frac{\lbd_\al-\lbd_\mu}{2b} + \ka_\mu +\cO(\vep),
$$
inserting this into the definition of $H$ one gets
$$
H = n\ka_\mu +\sum_{\al\neq\mu}^n \frac{\lbd_\al-\lbd_\mu}{2b}  +\cO(\vep),
$$
which substituting back leaves
$$
\ka_\mu(V^+_p(\vep)) = \frac{\lbd_\mu - a}{b(n+2)} -\sum_{\al\neq\mu}^n \frac{\lbd_\al-\lbd_\mu}{2b(n+2)} =\frac{1}{2b(n+2)}\left( -2a+(n+1)\lbd_\mu - \sum_{\al\neq\mu}^n\lbd_\al\right).
$$
The truncation error is given by the order of $\cO(\vep^{n+5})/b\sim\cO(\vep)$. Alternatively, one can solve the Hulin-Troyanov relation \ref{eq:volSph} to obtain a descriptor of $H$, and then use this into the expression of $\ka_\mu$ in terms of $\lbd_\mu$ and $H$ above.
\end{proof}

An analogous inversion process can be carried out with the series expansions of section \S\ref{sec:sphPatch}.

\begin{corollary}
	Denoting by $\lbd(p,\vep)\equiv\lbd(D_p(\vep)), V_p(\vep)\equiv V(D_p(\vep))$ the integral invariants of the hypersurface patch domain, then the corresponding curvature descriptors at scale $\vep>0$ and point $p\in\cS$, for any $\mu=1,\dots,n$, are
\begin{align}
	\cR(D^+_p(\vep)) & = 2(n+2)^2(n+4)\frac{\lbd_{n+1}(p,\vep)}{n\,\vep^4\, V_n(\vep)} - \frac{8(n+1)(n+2)}{n\,\vep^2}\left(\frac{V_p(\vep)}{V_n(\vep)} - 1 \right) \\[3mm]
	H(D^+_p(\vep)) & = (\pm)\sqrt{ 4(n+2)^2(n+4)\frac{\lbd_{n+1}(p,\vep)}{n\,\vep^4 V_n(\vep)} +\frac{8(n+2)^2}{n\,\vep^2}\left(1-\frac{V_p(\vep)}{V_n(\vep)} \right) }, \\[3mm]
	\ka_\mu(D^+_p(\vep)) & = \frac{2(n+2)}{\vep^2 H(D^+_p(\vep))}\left[ \frac{V_p(\vep)}{V_n(\vep)}+\frac{n+4}{\vep^2}\left( \frac{\vep^2}{n+2}-\frac{\lbd_\mu(p,\vep)}{V_n(\vep)} \right) -1 \right],
\end{align}
where the overall sign can be chosen by fixing a normal orientation from $$(\pm)=\text{\emph{sgn}}\langle\,\bE_{n+1}(D_p(\vep)),\,\bb(D_p(\vep))\,\rangle .$$ The eigenvectors $\bE_\mu(D_p(\vep))$ and $\bE_{n+1}(D_p(\vep))$ are descriptors of the principal and normal directions respectively. The corresponding errors are $|H^2(p)-H(D_p(\vep))^2|\leq\cO(\vep)$, $|\cR(p)-\cR(D_p(\vep)) |\leq\cO(\vep)$, and $|\ka^2_\mu(p)-\ka_\mu(D_p(\vep))^2|\leq\cO(\vep)$.
\end{corollary}
\begin{proof}
By solving the second term in equations \ref{eq:volPatch} and \ref{eq:eigen2Patch}, let us define
\begin{align*}
	A & = \frac{8(n+2)}{\vep^2}\left(\frac{V_p(\vep)}{V_n(\vep)} - 1 \right) + \cO(\vep), \\[1mm]
	B & = 2(n+2)(n+4)\frac{\lbd_{n+1}(p,\vep)}{\vep^4\, V_n(\vep)} + \cO(\vep),
\end{align*}	
so that we have the system of equations $A=H^2-2\cR,\; B=\frac{n+1}{n+2}H^2-\cR$ whose solution is
\begin{align*}
	\cR & = \frac{1}{n}((n+2)B - (n+1)A), \\
	H^2 & = \frac{(n+2)}{n}(2B - A).
\end{align*}
We can approximate the normal direction and orientation by using $\bE_{n+1}(p,\vep)$, and since the barycenter \ref{eq:patchBary} has normal component with leading order in terms of $H$, their mutual projection can serve to fix the orientation and overall relative sign of all the principal curvatures. The principal curvatures themselves are then solved from eq. \ref{eq:eigen1Patch} substituting the value of $H$ above, resulting in $\displaystyle\ka_\mu=\frac{1}{4H}(A-\Gamma_\mu),$ where $$\Gamma_\mu=\frac{8(n+2)(n+4)}{\vep^4}\left(\frac{\lbd_\mu(p,\vep)}{V_n(\vep)}-\frac{\vep^2}{n+2} \right)+\cO(\vep).$$ The errors follow straightforwardly by the truncation of $A,\, B,\,\Gamma_\mu$.
\end{proof} 

In the spirit of the limit formula obtained in \cite{alvarez2017} for regular curves in $\RR^n$, relating ratios of the covariance eigenvalues to the Frenet-Serret curvatures, we also state here analogous expressions for hypersurfaces using the ratios of the covariance eigenvalues, whose proofs are straightforward.

\begin{corollary}
	Let $p\in\cS$ and consider the spherical component invariants. Then for any $\mu,\nu=1,\dots,n$, the first $n$ eigenvalues, $\lbd_\mu(p,\vep)\equiv\lbd_\mu(V^+_p(\vep))$, of the covariance matrix $C(V^+_p(\vep))$ satisfy the following limit ratio:
\begin{equation}
	\lim_{\vep\rightarrow 0^+}\frac{V^2_{n+1}(\vep)}{V_n(\vep)}\frac{\lbd_\mu(p,\vep) -\lbd_\nu(p,\vep)}{\lbd_\mu(p,\vep)\lbd_\nu(p,\vep)} = \frac{4(n+3)^2}{(n+2)(n+4)}[\ka_\nu(p)-\ka_\mu(p)].
\end{equation}
\end{corollary}
\vspace{0.25cm}
\begin{corollary}
	Let $p\in\cS$ and consider the hypersurface patch invariants. Then for any $\mu,\nu=1,\dots,n$, the first $n$ eigenvalues, $\lbd_\mu(p,\vep)\equiv\lbd_\mu(D_p(\vep))$, of the covariance matrix $C(D_p(\vep))$ satisfy the following limit ratio:
\begin{equation}
	\lim_{\vep\rightarrow 0^+}V_n(\vep)\frac{\lbd_\mu(p,\vep) -\lbd_\nu(p,\vep)}{\lbd_\mu(p,\vep)\lbd_\nu(p,\vep)} = \frac{n+2}{2(n+4)}[\ka_\nu(p) -\ka_\mu(p)]H(p),
\end{equation}
	and the last eigenvalue satisfies:
\begin{equation}
	\lim_{\vep\rightarrow 0^+}V_n(\vep)\frac{\lbd_{n+1}(p,\vep)}{\lbd_\mu(p,\vep)\lbd_\nu(p,\vep)} = \frac{n+2}{2(n+4)}\left[\frac{n+1}{n+2}H^2(p) -\cR(p)\right].
\end{equation}
\end{corollary}

These ratios can be used as well to define descriptors solving for the curvature variables aided by the volume descriptor, like in the preceding corollaries.

%%%%%%%%%%%%%%%%%%%%%%%%%%%%%%%%%%%%%%%%%%%%%%%%%%%%%%%%%%%%%%%%%%%%%%
%%%%%%%%%%%%%%%%%%%%%%%%%%%%%%%%%%%%%%%%%%%%%%%%%%%%%%%%%%%%%%%%%%%%%%

%%%%%%%%%%%%%%%%%%%%%%%%%%%%%%%%%%%%%%%%%%%%%%%%%%%%%%%%%%%%%%%%%%%%%%
%%%%%%%%%%%%%%%%%%%%%%%%%%%%%%%%%%%%%%%%%%%%%%%%%%%%%%%%%%%%%%%%%%%%%%
%%%%%%%%%%%%%%%        6. CONCLUSIONS
%%%%%%%%%%%%%%%%%%%%%%%%%%%%%%%%%%%%%%%%%%%%%%%%%%%%%%%%%%%%%%%%%%%%%%
%%%%%%%%%%%%%%%%%%%%%%%%%%%%%%%%%%%%%%%%%%%%%%%%%%%%%%%%%%%%%%%%%%%%%%

\section{Conclusions}
\label{sec:conclusions}

In this paper we have generalized major PCA methods and results known for surfaces in space to establish the asymptotic relationship between integral invariants and the principal curvatures and principal directions of hypersurfaces of any dimension, which furnishes a method to obtain geometric descriptors at any given scale using the eigenvalue decomposition of the covariance matrix. We have seen that these methods are sufficient to provide also estimators of the Riemann curvature tensor of embedded submanifolds of higher-codimension, using its hypersurface projections onto the linear subspaces in ambient space spanned by the tangent space and each of the normal vectors from an orthonormal basis. These results establish a theoretical ground for the usage of the computational integral invariant approach to study the geometry of point clouds of high dimensionality, which can become a very valuable tool for Manifold Learning and Geometry Processing in a general setting.

%%%%%%%%%%%%%%%%%%%%%%%%%%%%%%%%%%%%%%%%%%%%%%%%%%%%%%%%%%%%%%%%%%%%%%
%%%%%%%%%%%%%%%%%%%%%%%%%%%%%%%%%%%%%%%%%%%%%%%%%%%%%%%%%%%%%%%%%%%%%%
%%%%%%%%%%%%%%%        A. APPENDIX: INTEGRALS
%%%%%%%%%%%%%%%%%%%%%%%%%%%%%%%%%%%%%%%%%%%%%%%%%%%%%%%%%%%%%%%%%%%%%%
%%%%%%%%%%%%%%%%%%%%%%%%%%%%%%%%%%%%%%%%%%%%%%%%%%%%%%%%%%%%%%%%%%%%%%

\appendix
\section{Integration of Monomials over Spheres}\label{sec:appendix}
Let $\bx =(x_1,\dots,x_n )\in\RR^n$, and denote the sphere and ball of radius $\vep$ in $\RR^n$ by:
$$
\SS^{n-1}(\vep)=\{\bx\in\RR^n : \|\bx\|=\vep \},\quad B^{n}(\vep)=\{\bx\in\RR^n : \|\bx\|\leq\vep \},
$$
where we set $\SS^{n-1}=\SS^{n-1}(1)$. Using generalized spherical coordinates $(r,\phi_1,\dots,\phi_{n-1})$, where $r=\|\bx\|,\; \cx_\mu=x_\mu / r\in\SS^{n-1}$, i.e.,
$$
\cx_1=\cos\phi_1,\;\dots,\quad\cx_{n-1}=\sin\phi_1\cdots\sin\phi_{n-2}\cos\phi_{n-1},\quad \cx_n=\sin\phi_1\cdots\sin\phi_{n-2}\sin\phi_{n-1},
$$
the Euclidean measure over the unit sphere and ball of any radius can be written as 
\begin{equation}\label{eq:cylMeasure}
	d\,\SS^{n-1}=d\phi_{n-1}\prod_{\mu=1}^{n-2}\sin^{n-1-\mu}(\phi_\mu)d\phi_\mu,\quad\quad d^nB=dx_1\cdots dx_n=r^{n-1}dr\;d\,\SS^{n-1}.
\end{equation}

\begin{definition}
 For any integers $\al_1,\dots,\al_n\in\{0,1,2,\dots\}$, the integrals of the monomials $x_1^{\al_1}\cdots x_n^{\al_n}$ over the unit sphere and the ball of radius $\vep$ are denoted by:
\begin{equation}\label{eq:constants}
	C^{(n)}_{\al_1\dots\al_n}=\int_{\SS^{n-1}} x_1^{\al_1}\cdots x_n^{\al_n}\; d\,\SS^{n-1},\quad\quad
	D^{(n)}_{\al_1\dots\al_n}=\int_{B^n(\vep)} x_1^{\al_1}\cdots x_n^{\al_n}\; d^nB.
\end{equation}
\end{definition}

These can be computed directly in spherical coordinates by collecting factors and separating the integrals into a product of integrals of powers of sines and cosines which can be given in terms of the Beta function, that then telescopes and simplifies; other shorter proof uses the usual exponential trick, see for example \cite{folland2001}, resulting in the following formula.

\begin{theorem}\label{th:folland}
	Denoting $\bet_\mu=\frac{1}{2}(\al_\mu+1)$, the values of the integrals \ref{eq:constants} over spheres are
\begin{equation}
	C^{(n)}_{\al_1\dots\al_n}=\begin{cases}
		0, & \text{ if some $\al_\mu$ is odd,} \\ \displaystyle
		2\frac{\Gamma(\bet_1)\Gamma(\bet_2)\cdots\Gamma(\bet_n)}{\Gamma(\bet_1+\bet_2+\cdots +\bet_n)}, & \text{ if all $\al_\mu$ are even},
	\end{cases}
\end{equation}
and the integrals over balls become
\begin{equation}
	D^{(n)}_{\al_1\dots\al_n}=\frac{\vep^{n+(\al_1+\cdots+\al_n)}}{n+(\al_1+\cdots+\al_n)}\; C^{(n)}_{\al_1\dots\al_n}.
\end{equation}
\end{theorem}

Notice that the values of the integrals of these monomials only depend on the combination of powers, not on which particular coordinates have those powers. Using these formulas we compute the relevant integrals that are needed for our work.
\begin{remark}
	Unless integrals over spheres of different dimension appear in the same expression, we shall abbreviate and omit the superscript $^{(n)}$ to be understood from the context.
\end{remark}

\begin{example}\label{ex:constants}
	Using the factorial property of the gamma function, $\Gamma(z+1)=z\Gamma(z)$, and the value $\Gamma(\frac{1}{2})=\sqrt{\pi}$, the integrals of monomials of even powers of order 2, 4 and 6, have the following relations (shortening $d\,\SS^{n-1}$ as $d\,\SS$):
\begin{align*}
	& C_{2} = \int_{\SS^{n-1}} x_1^{2}\; d\,\SS= 2\frac{\Gamma(\frac{3}{2})\Gamma(\frac{1}{2})^{n-1}}{\Gamma(\frac{3}{2}+\frac{n-1}{2})}=\frac{\pi^{n/2}}{\Gamma(\frac{n}{2}+1)}, \\[1mm]
	& C_{22}= \int_{\SS^{n-1}} x_1^{2}x_2^2\; d\,\SS = \frac{1}{n+2}\, C_2, \\[1mm]
	& C_{4}= \int_{\SS^{n-1}} x_1^{4}\; d\,\SS = \frac{3}{n+2}\, C_2 = 3\, C_{22}, \\[1mm]
	& C_{222}= \int_{\SS^{n-1}} x_1^{2}x_2^2x_3^2\; d\,\SS = \frac{1}{(n+2)(n+4)}\, C_2, \\[1mm]
	& C_{24}= \int_{\SS^{n-1}} x_1^{2}x_2^4\; d\,\SS = \frac{3}{(n+2)(n+4)}\, C_2 = 3\, C_{222}, \\[1mm]
	& C_{6}= \int_{\SS^{n-1}} x_1^{6}\; d\,\SS = \frac{15}{(n+2)(n+4)}\, C_2 = 15\, C_{222}.
\end{align*}
The value of $C_2$ is related to the $n$-dimensional volume of the ball of radius $\vep$, and the $(n-1)$-dimensional area of the unit sphere by
\begin{equation*}
	V_n(\vep)=\text{Vol}(B^n(\vep)) = \vep^n\, C_2,\quad\quad S_{n-1}=\text{Area}(\SS^{n-1})= n\, C_2.
\end{equation*}
The integrals over balls needed in our work are:
\begin{align*}
	& D_2 = \int_{B^n(\vep)} x_1^{2}\;\; dx_1\cdots dx_n = \frac{\vep^{n+2}}{n+2}C_2 = \frac{\vep^2}{n+2}V_n(\vep) , \\
	& D_{22}  = \int_{B^n(\vep)} x_1^{2}x_2^2\;\; dx_1\cdots dx_n = \frac{\vep^{n+4}}{(n+2)(n+4)}C_2= \frac{\vep^4}{(n+2)(n+4)}V_n(\vep) , \\
	& D_4  = \int_{B^n(\vep)} x_1^{4}\;\; dx_1\cdots dx_n = \frac{3\;\vep^{n+4}}{(n+2)(n+4)}C_2 = \frac{3\,\vep^4}{(n+2)(n+4)}V_n(\vep).
\end{align*}
\end{example}

We also need the integral of monomials over half-balls $B^+(\vep)$ (without loss of generality we can consider the half-ball is defined by $x_1\geq 0$). If all the $\al_i$ are even then nothing changes in the proof of theorem \ref{th:folland} except that now we integrate over half the domain and an extra factor of $\frac{1}{2}$ is needed. If any $\al_i$ is odd for $i\neq 1$, the integration over those variables is still carried out over the same domain so the overall integral is still 0. However, if $\al_1$ is odd the corresponding integral of that coordinate does not cancel out, and the main formula still holds with $\bet_1=1$ but without the factor of 2.
\begin{example}
Using the formula in the mentioned adjusted form, we define and compute
$$
	D^{(n)}_1 = \int_{B^+(\vep)} x_1\;\; dx_1\cdots dx_{n} =\frac{\vep^{n+1}\,\pi^{\frac{n-1}{2}}}{2\Gamma(\frac{n+3}{2})},
$$
which gives the constant needed in our main text
$$
	D^{(n+1)}_1 = \int_{B^+(\vep)} x_1\;\; dx_1\cdots dx_{n+1} = \frac{\vep^2}{n+2}V_n(\vep).
$$
When integrating $\int_{B^+(\vep)} x_1^2\;\;\text{dVol}$, we shall just write $\frac{D_2}{2}$ to be consistent with our notation above.
\end{example}

%%%%%%%%%%%%%%%%%%%%%%%%%%%%%%%%%%%%%%%%%%%%%%%%%%%%%%%%%%%%%%%%%%%%%%
%%%%%%%%%%%%%%%%%%%%%%%%%%%%%%%%%%%%%%%%%%%%%%%%%%%%%%%%%%%%%%%%%%%%%%

\addtocontents{toc}{\SkipTocEntry}
\section*{Acknowledgments}
We would like to thank Louis Scharf for very helpful discussions. J.\'A.V. would like to thank Miguel Dovale \'Alvarez for many useful discussions during the writing of this paper.

\bibliographystyle{amsplain}
\bibliography{AKP-Principal-Curvatures}

\end{document}